\documentclass{amsproc}%
\usepackage{eurosym}
\usepackage{amssymb}
\usepackage{amsfonts}
\usepackage{amsmath}
\usepackage{graphicx}%
\theoremstyle{plain}

\newtheorem{corollary}{Corollary}

\newtheorem{definition}{Definition}
\newtheorem{example}{Example}

\newtheorem{lemma}{Lemma}

\newtheorem{proposition}{Proposition}

\newtheorem{theorem}{Theorem}
\numberwithin{equation}{section}
\begin{document}
\title[Equivalence of rational links and 2-bridge links revisited]{Equivalence of rational links and 2-bridge links revisited}
\author{M. M. Toro}
\curraddr{Universidad Nacional de Colombia, Sede Medell\'{\i}n}
\email{mmtoro@unalmed.edu.co}
\thanks{Facultad de Ciencias, Universidad Nacional de Colombia}
\date{June 2014}

\begin{abstract}
In this paper we give a simple proof of the equivalence
between the rational link associated to the continued fraction $\left[
a_{1},a_{2},\cdots a_{m}\right] ,$ $a_{i}\in\mathbb{N}$, and the two bridge
link of type $p/q,$ where $p/q$ is the rational given by $\left[ a_{1}%
,a_{2},\cdots a_{m}\right]  $. The known proof of this equivalence relies on
the two fold cover of a link and the classification of the lens spaces. Our
proof is elementary and combinatorial and follows the naive approach of
finding a set of movements to transform the rational link given by $\left[
a_{1},a_{2},\cdots a_{m}\right]  $ into the two bridge link of type $p/q$.
\end{abstract}
\maketitle
\section{Introduction}	
The equivalence between the two bridge link of type $p/q$, introduced and
classified by Schubert \cite{Sc}, and the rational links associated to a
continued fraction $\left[  a_{1},a_{2},\cdots a_{m}\right]  $, introduced by
Conway \cite{Co}, is one example of the beautiful relations between knot
theory and other mathematical subjects, in this case number theory. In an
elementary course of knot theory, this relation captures the students
attention and imagination, but the known proof requires advanced techniques
from 3-manifold theory that are out of reach at that level. For this reason,
we seek an elementary proof, that follows the naive approach of finding an
algorithm to change one of the diagrams into the other.

In this paper we will transform the Conway diagram $C$ of a rational link,
associated to a continued fraction $\left[  a_{1},a_{2},\cdots a_{m}\right]
$, with $a_{i}\in\mathbb{N}^{+},$ into a diagram $S\ $of the two bridge link
of type $p/q$, where $p/q$ is the rational given by the continued fraction
$\left[  a_{1},a_{2},\cdots a_{m}\right]  $. In this way we give an elementary
proof of the equivalence between the rational link $\left[  a_{1},a_{2},\cdots
a_{w}\right]  $ and the two bridge link of type $p/q.$ This equivalence can be
found in \cite{BuZi}, but this proof requires to consider the two fold cover
of the link and the classification of the lens spaces. Our proof differs
completely of this approach, and instead, use the direct method of finding a
set of moves, that can be described in a recursive way as a sequence of steps.
In each step of the transformation process we are able to describe precisely
the changes in the diagram and to produce a sequence of integers that keep
track of the changes and will allow us to confirm that at the end of the
transformation we obtain the diagram of the 2-bridge link of the right type.
The result of this paper together with the results in \cite{kl} and
\cite{klt}, the complete classification of rational links and two bridge links
is completed, without requiring advanced techniques of three dimensional topology.

The role of knot theory as a didactic tool, not only in undergraduate courses
but also as a subject to develop mathematical awareness in high school
students, requires some effort to present part of the theory in an appropriate
level. This elementary proof of the equivalence of the two families of links
follows this approach.

The technique used in the transformation of a rational diagram into a 2-bridge
diagram can be implemented in the transformation of a link diagram given by a
6-plat into a 3-bridge diagram given by the Schubert presentation as a triple
$\left(  p/n,q/m,s/l\right)  $, where $p,n,q,m,s,l$ are integers, see
\cite{HMTT}. This is a subject of current research.

In section 1 we describe the diagrams of rational links and 2-bridge links and
fix some notation. In section 2 we introduce the transformation process and
establish the main result. Section 3 has a collection of technical results, of
combinatorial and computational nature, that will be used in Section 4 to
prove the main result.

\section{Diagrams and notation}

Let us consider the continued fraction $\left[  a_{1},a_{2},\cdots
a_{m}\right]  $, with $a_{i}$ a positive integer, $1\leq i\leq m.$ Figure
\ref{fig1} shows the diagrams $C=C\left[  a_{1},a_{2},\cdots a_{m}\right]  $
of rational links associated to a continued fraction $\left[  a_{1}%
,a_{2},\cdots a_{m}\right]  $, where the tangles in even positions are
positive and the ones in odd positions are negative, see Fig. \ref{fig1} a.%

\begin{figure}[h]%
\centering
\includegraphics[
height=2.5486in,
width=4.1485in
]%
{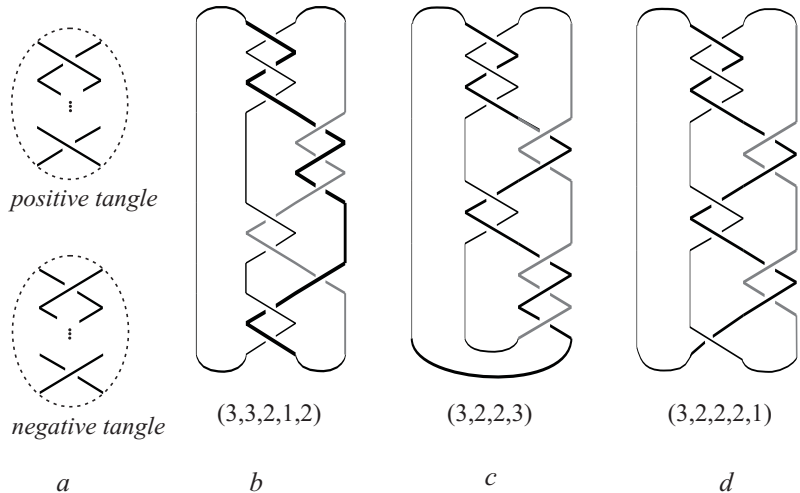}%
\caption{Rational links diagrams, with odd and even number of tangles.}%
\label{fig1}%
\end{figure}
The final arcs that close the diagram depend on $m.$ If $m$ is odd, we use the
form shown in Fig. \ref{fig1} b., and if it is even we use the form shown in
Fig. \ref{fig1} c. Usually we always can consider $m$ odd, by a simple
deformation of the diagram, as shown in Fig. \ref{fig1} d., see \cite{kl}, but
in our work we need to consider both diagram types. Note that our convention
for rational links follows \cite{Mu}, \cite{Gr} and \cite{kl} and differs from
the standard one given in \cite{BuZi}, \cite{Co}, \cite{Ka} and \cite{Wi}, so
our rational link $C\left[  a_{1},a_{2},\cdots a_{m}\right]  $ corresponds to
the mirror image of the one in \cite{BuZi}.

To a rational number $p/q,$ with $p$ and $q$ relatively primes and $0<q<p,$ we
associate a two bridge diagram $S=S\left(  p/q\right)  $ as shown in Fig.
\ref{fig2}. We call $\alpha$ the bridge to the left and $\beta$ the bridge to
the right. The integer $p$ represents the number of segments in which we
divided each one of the bridges, so the number of crossings under each bridge
is $p-1$. The integer $q$ represents the position of the first crossing of the
bridge $\alpha$ under the bridge $\beta$, as shown in Fig. \ref{fig2} a.,
counting in a clockwise direction. See \cite{Ka} and \cite{Sc} for a more
detailed description of the 2-bridge diagram.%

\begin{figure}[h]%
\centering
\includegraphics[
height=1.4667in,
width=4.9552in
]%
{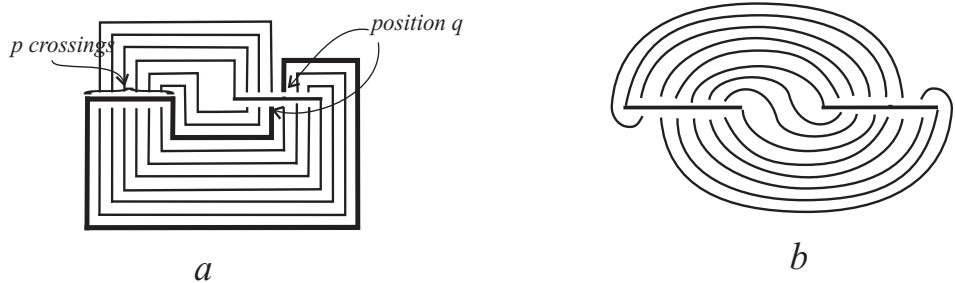}%
\caption{Two bridge link diagrams $p/q.$ Particular case with $p=7$ and
$q=3$.}%
\label{fig2}%
\end{figure}
Our 2-bridge link corresponds to the mirror image of the 2-bridge link in
\cite{Sc} and \cite{BuZi}. Usually the standard diagram of a 2-bridge link is
symmetric, as shown in Fig. \ref{fig2} b., but for our purpose we will
consider the bridge $\beta$ as formed by two segments, so the diagram in Fig.
\ref{fig2} a. is more appropriate.

We will describe a process to transform the diagram $C\left[  a_{1}%
,a_{2},\cdots a_{m}\right]  $ into the diagram $S\left(  p/q\right)  $, when
$p/q$ is the rational given by the continued fraction $\left[  a_{1}%
,a_{2},\cdots,a_{m}\right]  $. The process will be defined in a recursive way.
In each step $n,$ with$~1\leq n\leq m,$ where $m$ is the length of the
continued fraction, we construct a sequence of integers $p_{n}$ and $q_{n}$
and we prove that they satisfy the recurrence relation of the convergents of
the continued fraction $\left[  a_{1},a_{2},\cdots,a_{m}\right]  $, see
\cite{Bu},
\begin{equation}
p_{n+1}=a_{n+1}p_{n}+p_{n-1},\ \ \ q_{n+1}=a_{n+1}q_{n}+q_{n-1,}%
\ \ \ \ \ \ n\geq1\text{,} \label{recfraccion}%
\end{equation}
with\ $p_{0}=1,\ p_{1}=a_{1},\ q_{0}=0,q_{1}=1.$

\textbf{Notation:} In all the paper, each $a_{i}$ will be a positive integer.
For a real number $a,$ $\left\lceil a\right\rceil $, the ceiling of $a$, will
be the least integer greater or equal to $a,$ and $\left\lfloor a\right\rfloor
$, the floor of $a$, will be the greatest integer less or equal to $a.$ For an
integer $a$ we define%
\begin{equation}
\mu_{a}=\left\{
\begin{array}
[c]{cc}%
0 & \text{if }a\text{ is even}\\
1 & \text{if }a\text{ is odd.}%
\end{array}
\right.  \label{defmu}%
\end{equation}
We will use also the Kronecker delta $\delta_{lt}$ to indicate $1$ if $l=t$
and $0$ otherwise.

\section{Geometric description of the transformation process}

To transform the diagram $C$ of a rational link into a two bridge diagram $S$,
we take the two bridges as the two superior arcs of diagram $C$. We call
$\alpha$ the bridge to the left and $\beta$ the one to the right. From the two
bridges emerge 4 {\normalsize strings}, that wave to form the tangles that
conforms the rational link. The {\normalsize string} to the right does not
play any role in the waving, so we will consider only 3 {\normalsize strings}
and we take bridge $\alpha$ as formed by {\normalsize string} 1 and divide
bridge $\beta$ into {\normalsize strings} 2 and 3, so we will consider $\beta$
as formed by two independent {\normalsize strings}, see Fig. \ref{fig3} a.

The transformation process will be a sequence of steps, one for each one of
the tangles that form the rational link. In each step we transform the tangle
so that all the crossings will be under the bridges. In each step of our
process each {\normalsize string} will play a different role. Two of them,
that we will call the \textit{guides, }are the ones that form the tangle and
will be modified. The third one will be idle.

\begin{definition}
\label{defguide}We define as \textit{principal guide }in step $n$ the string
that crosses over the other one in the first crossing of the tangle $a_{n},$
and we define as secondary guide the string that crosses under.
S{\normalsize ee Fig. \ref{fig4} a}.\textit{\ }The principal and secondary
guides in step $n$ are denoted by $i_{n}$ and $j_{n}$ respectively. The idle
{\normalsize string is the one }that does not form part of the tangle $a_{n}$
and is denoted $k_{n}.$
\end{definition}

To simplify notation we will drop the subscript and\ use only $i,j$ and $k$
when there is no place to confusion.

In order to count the number of crossings under each bridge we will keep track
of the {\normalsize strings} that form the crossings and we will count the
number of times that {\normalsize string} $r$ crosses under
{\normalsize string} $t,$ $1\leq r,t\leq3.$ This information is stored in a
data matrix $D_{n}=\left[  d_{rt}\left(  n\right)  \right]  _{r=1,3}^{t=1,3},$
where $d_{rt}\left(  n\right)  $ represents the times that
{\normalsize string} $r$ crosses under {\normalsize string} $t$. In each step
of our process we transform the data matrix $D_{n}$ into the matrix $D_{n+1}$.

Let us take the rational diagram $C$ described by the continued fraction
$\left[  a_{1},\cdots,a_{m}\right]  $. We start the process by taking
\begin{equation}
D_{0}=\left[  0\right]  \label{d0}%
\end{equation}
and describe a recurrence process in order to change $C$ into $S$.

\subsection{Step One}

In the first step we will unravel the first tangle $a_{1}$ as shown in Fig.
\ref{fig3} a. The tangle is formed by the 1 and 2 {\normalsize strings}, with
1 as the principal guide, 2 as the secondary guide and $3$ as the idle
{\normalsize string}. The movement changes the tangle with $a_{1}$ positive
crossings into a 2-tangle in which there are only crossings under the bridges.
The movement produces $a_{1}$ crossings under {\normalsize string}$\ 1$,
$\left\lfloor a_{1}/2\right\rfloor $ of them formed by undercrossings of
{\normalsize string} $1$ and $\left\lceil a_{1}/2\right\rceil $ corresponding
to crossings of {\normalsize string} $2$. Under {\normalsize string} $2$ we
will have $a_{1}-1\ $crossings, $\left\lceil \left(  a_{1}-1\right)
/2\right\rceil $ of them formed by {\normalsize string} $1$ and $\left\lfloor
\left(  a_{1}-1\right)  /2\right\rfloor $ by {\normalsize string }%
$2$.{\normalsize

\begin{figure}[h]%
\centering
\includegraphics[
height=1.254in,
width=4.6307in
]%
{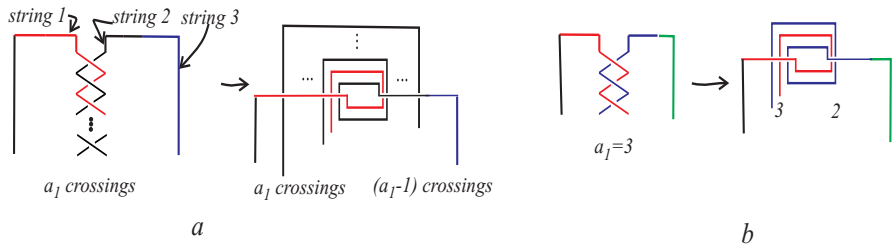}%
\caption{Step one, only {\protect\normalsize strings} 1 and 2 are used. Fig. b
shows the case $a_{1}=3$.}%
\label{fig3}%
\end{figure}
}

Therefore, we have the following matrix to describe step one%
\[
D_{1}=\left[  d_{rt}\left(  1\right)  \right]  _{r=1,3}^{t=1,3}=%
\begin{bmatrix}
\left\lfloor \frac{a_{1}}{2}\right\rfloor \medskip & \left\lceil \frac
{a_{1}-1}{2}\right\rceil  & 0\\
\left\lceil \frac{a_{1}}{2}\right\rceil \medskip & \left\lfloor \frac{a_{1}%
-1}{2}\right\rfloor  & 0\\
0 & 0 & 0
\end{bmatrix}
.
\]

The total number of crossings under {\normalsize string} $t$ is given by the
sum $d_{1t}+d_{2t}+d_{3t},$ so these are the numbers that allow us to describe
the 2-bridge link produced at the end of the process.

\begin{definition}
\label{defst}For $t\in\{1,2,3\}$ we define $s_{n,t}$ as the sum of the
$t\ $column of matrix $D_{n}$, i.e., $s_{n,t}=d_{1t}+d_{2t}+d_{3t}$.
\end{definition}

\subsection{Step n+1}

Let us suppose that we have changed the tangles $a_{1},\cdots,a_{n}$ and the
information is stored in $D_{n}$.

In step $n+1$ we change the tangle $a_{n+1}$ into a rational tangle, in the
same way as we proceeded in step one. We have $a_{n+1}$ crossings under the
principal guide and $a_{n+1}-1$ under the secondary guide, see Fig.
{\normalsize \ref{fig4} b. }We move all the crossings, forming a parallel set
of {\normalsize strings}, followings the direction of each of the guides,
ending the movement when we reach the initial points of the guides. In this
way all the crossings are under the bridges, see Fig. {\normalsize \ref{fig4}.
}Figure \ref{fig6} shows the transformation process of $C[1,2,2].$ In Fig.
\ref{fig6}b. we made the first step and transform the tangle 1. In Fig.
\ref{fig6}c. we change the tangle 2 and in Fig. \ref{fig6}d. we move all the
crossing, following the guides.%

\begin{figure}[h]%
\centering
\includegraphics[
height=1.5446in,
width=4.734in
]%
{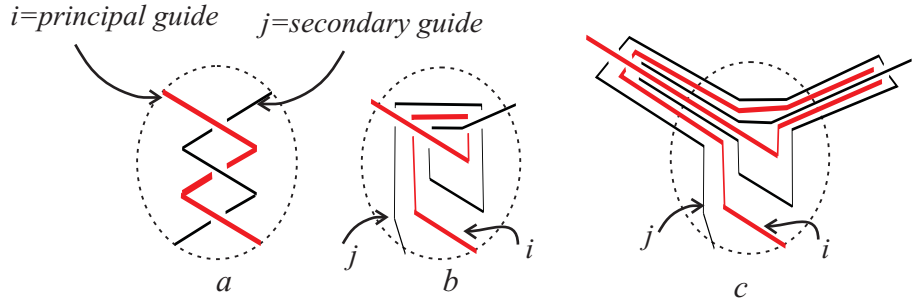}%
\caption{Chages in step n+1.}%
\label{fig4}%
\end{figure}

Now, the important part is to keep track of the number of crossings we have at
the end of step $n+1$ and to find the new guides.

To simplify notation, take $a=$ $a_{n+1}$ and let $D=D_{n}$ $=\left[
d_{rt}\right]  \ $and $D^{\prime}=D_{n+1}$ $=\left[  d_{rt}^{\prime}\right]  $
defined below. Let us find the values of $d_{rt}^{\prime}$ in terms of
$d_{rt}$ and $a$, for $1\leq r,t\leq3.$

Suppose that in step$\ n+1$the guides are $i=i_{n+1}\ $and$\ j=j_{n+1}$ and
the idle {\normalsize string\ is }$k=k_{n+1}$. As the {\normalsize string} $k$
is idle, the values of $d_{kt},$ $1\leq t\leq3,$ corresponding to
undercrossings of {\normalsize string} $k,$ do not change, so $d_{kt}^{\prime
}=d_{kt},\ $for\ $1\leq t\leq3$.

Each time that {\normalsize string} $i$ crosses under {\normalsize string}
$t$, there are $2a$ new crossings under that {\normalsize string},
$2\left\lfloor a/2\right\rfloor $ of which are formed by {\normalsize string}
$i$, or principal guide and $2\left\lceil a/2\right\rceil $ are formed by
{\normalsize string} $j$, or secondary guide, see Fig. \ref{fig4} c.

Each time that {\normalsize string} $j$ crosses under {\normalsize string}
$t$, there are $2\left(  a-1\right)  $ new crossings under
{\normalsize string} $t,$ $2\left\lceil \left(  a-1\right)  /2\right\rceil $
corresponding to crossings of {\normalsize string} $i$ and $2\left\lfloor
\left(  a-1\right)  /2\right\rfloor $ corresponding to crossings of
{\normalsize string} $j$. See Fig. {\normalsize \ref{fig4} }c.

Besides these new crossings, when we arrive to the end of {\normalsize string}
$i$, we have $a$ additional crossings, $\left\lfloor a/2\right\rfloor $
corresponding to undercrossings of {\normalsize string} $i$ and $\left\lceil
a/2\right\rceil $ corresponding to undercrossings of {\normalsize string}
$\dot{j}$. At the end of {\normalsize string} $j,$ there are $a-1\ $additional
crossings, $\left\lceil \left(  a-1\right)  /2\right\rceil $ corresponding to
{\normalsize string} $i$ and $\left\lfloor \left(  a-1\right)  /2\right\rfloor
$ corresponding to {\normalsize string} $j$.

So at the end of step $n+1\ $we have
\begin{align}
d_{kt}^{\prime} &  =d_{kt},\ \ \text{for }1\leq t\leq3,\nonumber\\
d_{ik}^{\prime} &  =\left(  1+2\left\lfloor \frac{a}{2}\right\rfloor \right)
d_{ik}+2\left\lceil \frac{a-1}{2}\right\rceil d_{jk},\nonumber\\
d_{ij}^{\prime} &  =\left(  1+2\left\lfloor \frac{a}{2}\right\rfloor \right)
d_{ij}+2\left\lceil \frac{a-1}{2}\right\rceil d_{jj}+\left\lceil \frac{a-1}%
{2}\right\rceil ,\nonumber\\
d_{ii}^{\prime} &  =\left(  1+2\left\lfloor \frac{a}{2}\right\rfloor \right)
d_{ii}+2\left\lceil \frac{a-1}{2}\right\rceil d_{ji}+\left\lfloor \frac{a}%
{2}\right\rfloor ,\nonumber\\
d_{jk}^{\prime} &  =\left(  1+2\left\lfloor \frac{a-1}{2}\right\rfloor
\right)  d_{jk}+2\left\lceil \frac{a}{2}\right\rceil d_{ik},\nonumber\\
d_{jj}^{\prime} &  =\left(  1+2\left\lfloor \frac{a-1}{2}\right\rfloor
\right)  d_{jj}+2\left\lceil \frac{a}{2}\right\rceil d_{ij}+\left\lfloor
\frac{a-1}{2}\right\rfloor ,\nonumber\\
d_{ji}^{\prime} &  =\left(  1+2\left\lfloor \frac{a-1}{2}\right\rfloor
\right)  d_{ji}+2\left\lceil \frac{a}{2}\right\rceil d_{ii}+\left\lceil
\frac{a}{2}\right\rceil ,\label{reldprima}%
\end{align}
where $\left\{  i,j\right\}  \ $are the guides\ and $k$ is the idle
{\normalsize string}.

The process just described can be written as a recurrence relation in terms of
matrices, as we will describe in Theorem \ref{teorec}.

For the example $C\left[  1,2,2\right]  $ shown in Fig. \ref{fig6}, we have
\[
D_{1}=%
\begin{bmatrix}
0 & 0 & 0\\
1 & 0 & 0\\
0 & 0 & 0
\end{bmatrix}
,\ \
\begin{bmatrix}
s_{1,1}\\
s_{1,2}\\
s_{1,3}%
\end{bmatrix}
=%
\begin{bmatrix}
1\\
0\\
0
\end{bmatrix}
\ \ \ \ \ \ \ D_{2}=%
\begin{bmatrix}
0 & 0 & 1\\
1 & 0 & 0\\
1 & 0 & 1
\end{bmatrix}
,\ \
\begin{bmatrix}
s_{2,1}\\
s_{2,2}\\
s_{2,3}%
\end{bmatrix}
=%
\begin{bmatrix}
2\\
0\\
2
\end{bmatrix}
\
\]

\subsection{Final step}

At the end of step $m$ we have transformed all the tangles and we have
produced a 2-bridge link. When $m$ is even, we reach a reduced 2-bridge
diagram but when $m$ is odd, we reach a no reduced diagram with a kink formed
by bridge $\alpha$, that when simplified by a type I Reidemeister move
produces a reduced 2-bridge diagram, see Fig. \ref{fig5}.%

\begin{figure}[ptb]%
\centering
\includegraphics[
height=1.0776in,
width=5.2598in
]%
{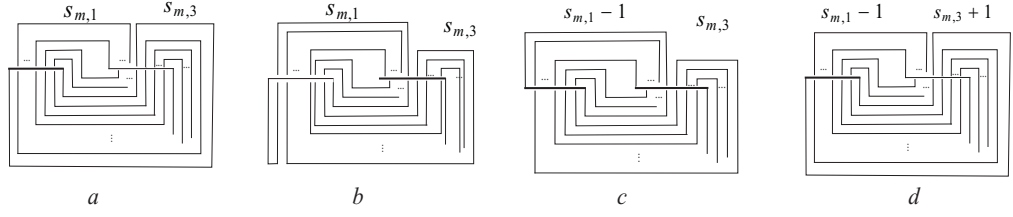}%
\caption{Final step. Figure a is for the case $m$ even. Figures b, c and d are
for $m$ odd.}%
\label{fig5}%
\end{figure}

Now we interpret the meaning of the values $s_{m,1}$ and $s_{m,3}$ in the
description of the final diagram.

When $m$ is even, see Fig. \ref{fig5} a., $s_{m,1}$ is the number of crossings
under bridge $\alpha$ and $s_{m,3}$ is the number of crossings under
{\normalsize string} $3,$ so the position of the crossing under which the
bridge $\alpha$ crosses bridge $\beta$ is $s_{m,3},$ therefore we obtain the
2-bridge link of type
\[
\left(  s_{m,1}+1\right)  /s_{3,m}.
\]
When $m$ is odd, see Fig. \ref{fig5} b., there is a kink formed by the bridge
$\alpha$ and after the simplification by type I Reidemeister move, see Fig.
\ref{fig5} c., the number of crossings under bridge $\alpha$ decreased by one
but the position of the crossing under which the bridge $\alpha$ crosses
bridge $\beta$ increases by one, see Fig. \ref{fig5} d., and now it is
$s_{m,3}+1$, so we obtain the 2-bridge link%
\[
s_{m,1}/\left(  s_{3,m}+1\right)  .
\]

In Fig. \ref{fig6} e. to f. we show the final step in the transformation of
$C[1,2,2].$ In Fig. \ref{fig6} e. we change the tangle 2 and in Fig.
\ref{fig6} f. we move all the crossings following the guides. In Fig.
\ref{fig6} g. it is clear the kink that is simplified in Fig. \ref{fig6} h.%

\begin{figure}[h]%
\centering
\includegraphics[
height=2.7717in,
width=3.7343in
]%
{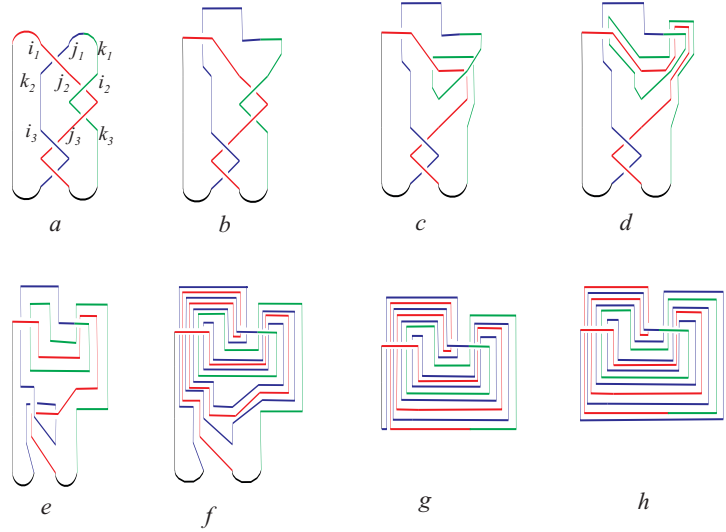}%
\caption{ Transformation of $C\left[  1,2,2\right]  $ into the 2-bridge knot
$7/5$.}%
\label{fig6}%
\end{figure}
For the example $C\left[  1,2,2\right]  $ we have
\[
\ D_{3}=%
\begin{bmatrix}
2 & 1 & 1\\
4 & 1 & 2\\
1 & 0 & 1
\end{bmatrix}
,\ \
\begin{bmatrix}
s_{3,1}\\
s_{3,2}\\
s_{3,3}%
\end{bmatrix}
=%
\begin{bmatrix}
7\\
2\\
4
\end{bmatrix}
\
\]
and after the simplification we get 6 crossings under string 1 and 5 crossings
under string 3, so we have the 2-bridge knot $7/5.$

\subsection{Main result}

In order to have the equivalence between the rational link $C\left[
a_{1},a_{2},\cdots a_{m}\right]  $ and the 2-bridge link $p/q,$ where $p/q$ is
the rational given by $\left[  a_{1},a_{2},\cdots a_{m}\right]  $, we need to
prove that the diagram obtained in the final step is in fact the diagram of
the 2-bridge link $p/q$.

\begin{definition}
\label{defpyq}Define $p_{n}=s_{n,1}+1-\mu_{n}$ and $q_{n}=s_{n,3}+\mu_{n},$
where $\mu_{n}$ was defined in (\ref{defmu}).
\end{definition}

So we need to prove that
\[
p_{m}/q_{m}=p/q,
\]
and we do so by proving the following theorem:

\begin{theorem}
[Main Result]Given the continued fraction $\left[  a_{1},a_{2},\cdots
a_{m}\right]  $ that represents the rational $p/q,$ the integers $p_{n}$ and
$q_{n}$ defined above satisfy the recurrence relations
\begin{align*}
p_{n+1}  &  =a_{n+1}p_{n}+p_{n-1},\ \text{for\ }1\leq n,\ \ p_{0}%
=1,\ p_{1}=a_{1},\\
q_{n+1}  &  =a_{n+1}q_{n}+q_{n-1},\ \text{for\ }1\leq n,\ \ q_{0}=0,q_{1}=1.
\end{align*}

Therefore $p_{n}/q_{n}$ is the $n$-convergent of the continued fraction
$\left[  a_{1},\cdots,a_{m}\right]  $ and so
\[
p_{m}/q_{m}=p/q\text{,}%
\]
then the rational link $C\left[  a_{1},a_{2},\cdots a_{m}\right]  $ is
equivalent to the 2-bridge link $p/q$.
\end{theorem}

The proof of this theorem is the content of the following sections.

\section{Algorithmic description of the transformation process}

In order to describe the transformation process the guides play a fundamental
role. To each step $n\ $we assign a permutation $\sigma_{n}\in$ $S_{3}$ that
keep track of the guides.

\begin{definition}
We define a permutation of $\left\{  1,2,3\right\}  $ by
\begin{equation}
\sigma_{n}\left(  1\right)  =i_{n},\sigma_{n}\left(  2\right)  =j_{n}%
=,\sigma_{n}\left(  3\right)  =k_{n},
\end{equation}
where $i_{n},j_{n}$ and $k_{n}$ were defined in Definition \ref{defguide},
$1\leq n\leq m.$
\end{definition}

\begin{lemma}
The permutation $\sigma_{n}$ satisfies the following recurrence relation%
\begin{equation}
\sigma_{n+1}=\left\{
\begin{array}
[c]{ll}%
\sigma_{n}\left(  1\ 3\right)  ,\ \ \ \ \ \ \  & \text{if }a_{n}\ \text{is
even,}\\
\sigma_{n}\left(  1\ 3\ 2\right)  , & \text{if }a_{n}\ \text{is odd,}%
\end{array}
\right.  \ \ \ \text{for }n>1\text{ and }\sigma_{1}=Id.
\end{equation}

\end{lemma}

\begin{proof}
Just consider the diagrams in Fig. \ref{fig7}.%

\begin{figure}[h]%
\centering
\includegraphics[
height=1.375in,
width=4.7072in
]%
{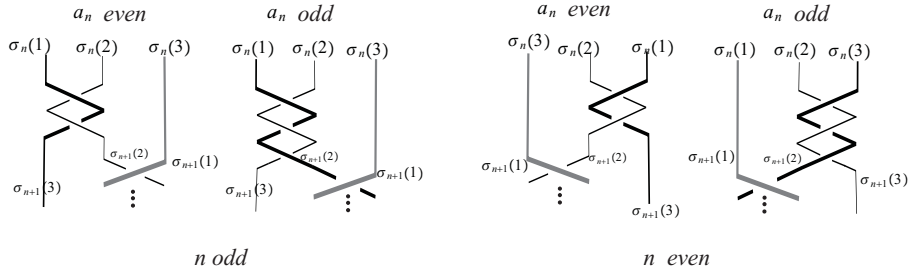}%
\caption{Changes from permutation $\sigma_{n}$ to permutation $\sigma_{n+1}.$}%
\label{fig7}%
\end{figure}

\end{proof}

\begin{corollary}
\label{reglaguias}For $n>1$, $i_{n+1}=k_{n}$ and if $a_{n}\ $is even then
$j_{n+1}=j_{n}\ $and $k_{n+1}=i_{n}$ and if $a_{n}\ $is odd then $j_{n+1}=$
$i_{n}$ and $k_{n+1}=j_{n}$.
\end{corollary}

\begin{definition}
\label{defmatrizm}For each step $n$, with guides $i_{n}=i,j_{n}=j$ and tangle
$a_{n},$ we define the matrix $M_{n}=M\left(  a_{n},i,j\right)  =\left[
m_{rt}\right]  \in M_{3\times3}\left(  \mathbb{N}\right)  $ by
\[
m_{ii}=\left\lfloor a_{n}/2\right\rfloor ,\ m_{ij}=\left\lceil (a_{n}%
-1)/2\right\rceil ,\ m_{ji}=\left\lceil a_{n}/2\right\rceil ,\ m_{jj}%
=\left\lfloor (a_{n}-1)/2\right\rfloor
\]
and zero elsewhere.
\end{definition}

The following lemma is clear and give us an alternative way to compute $M_{n}
$.

\begin{lemma}
The matrix $M_{n}$ can be described as%
\[
M_{n}=P_{\sigma_{n}}^{{}}M\left(  a_{n}\right)  P_{\sigma_{n}}^{-1}.
\]
where $M\left(  a_{n}\right)  =%
\begin{bmatrix}
\left\lfloor a_{n}/2\right\rfloor  & \left\lceil (a_{n}-1)/2\right\rceil  &
0\\
\left\lceil a_{n}/2\right\rceil  & \left\lfloor (a_{n}-1)/2\right\rfloor  &
0\\
0 & 0 & 0
\end{bmatrix}
$ and $P_{\sigma_{n}}$ is the permutation matrix of $\sigma_{n}$.
\end{lemma}

\begin{theorem}
\label{teorec}The associated matrix $D_{n}$ of the transformation process of
changing the diagram $C$ of the rational link into a 2-bridge diagram
satisfies the recurrence relation%
\begin{equation}
D_{n+1}=2M_{n+1}D_{n}+D_{n}+M_{n+1},\ 0\leq n\leq m \label{recumat}%
\end{equation}
where $M_{n+1}$ is the changing matrix associated to $a_{n+1}$ and to the
permutation $\sigma_{n+1}$, $D_{0}=$ $\left[  0\right]  $ and $\sigma_{1}=Id$.
\end{theorem}

\begin{proof}
This is just a new way to write the relations given in (\ref{reldprima}).
\end{proof}

The recurrence described in Theorem \ref{teorec} is very easy to implement in
a computer, in our case we use \textit{Mathematica.} The following is an
example of the calculation to transform the rational diagram $C\left[
2,3,1,2,3\right]  $.

\begin{example}
\label{exam}Change the rational $C\left[  2,5,4,1\right]  $ into the 2-bridge
$57/26$.
\end{example}

\begin{center}%
\begin{tabular}
[c]{|l|l|l|l|l|l|l|l|}\hline
Step $n$ & $a_{n}$ & $i_{n}$ & $j_{n}$ & $M_{n}$ & $D_{n}$ & $p_{n}$ & $q_{n}%
$\\\hline
$1$ & $2$ & $1$ & $2$ & $\overset{\ \ \ \ }{%
\begin{bmatrix}
1 & 1 & 0\\
1 & 0 & 0\\
0 & 0 & 0
\end{bmatrix}
}_{\ \ _{\ \ }}$ & $%
\begin{bmatrix}
1 & 1 & 0\\
1 & 0 & 0\\
0 & 0 & 0
\end{bmatrix}
_{\ \ }$ & $2$ & $1$\\\hline
$2$ & $5$ & $3$ & $2$ & $\overset{\ \ \ }{%
\begin{bmatrix}
0 & 0 & 0\\
0 & 2 & 3\\
0 & 2 & 2
\end{bmatrix}
}_{\ \ _{\ \ }}$ & $%
\begin{bmatrix}
1 & 1 & 0\\
5 & 2 & 3\\
4 & 2 & 2
\end{bmatrix}
_{\ \ \ }$ & $11$ & $5$\\\hline
$3$ & $4$ & $1$ & $3$ & $\overset{\ \ \ }{%
\begin{bmatrix}
2 & 0 & 2\\
0 & 0 & 0\\
2 & 0 & 1
\end{bmatrix}
}_{\ \ _{\ \ }}$ & $%
\begin{bmatrix}
23 & 13 & 10\\
5 & 2 & 3\\
18 & 10 & 7
\end{bmatrix}
$ & $46$ & $21$\\\hline
$4$ & $1$ & $2$ & $3$ & $\overset{\ \ \ }{%
\begin{bmatrix}
0 & 0 & 0\\
0 & 0 & 0\\
0 & 1 & 0
\end{bmatrix}
_{\ \ _{\ \ }}}$ & $%
\begin{bmatrix}
23 & 13 & 10\\
5 & 2 & 3\\
28 & 15 & 13
\end{bmatrix}
$ & $57$ & $26$\\\hline
\end{tabular}

Table 1
\end{center}

\subsection{Properties of matrix $D_{n}$}

As Example \ref{exam} shows, there are some interesting patterns in the
entries of matrix $D_{n}$. For instance, the column 1 is "almost" the sum of
columns 2 and 3. In this section we will find relations among the entries of
matrix $D_{n}$ and among the entries of $D_{n}$ and $D_{n+1}$. The results of
this section will be used in the proof of the main result. All the results
depend on the guides $i_{n+1},j_{n+1}$ and the parity of $n\ $and$\ a_{n}$.

\begin{proposition}
\label{propsum}Given the rational link $C\left[  a_{1},a_{2},\cdots
,a_{m}\right]  $, the sequence of integers $s_{n,t}$ satisfies the recurrence
relation%
\[
s_{n+1,t}=a_{n+1}s_{n,t}+a_{n+1}\left(  d_{it}+d_{jt}-d_{kt}+\delta
_{it}+\delta_{jt}\right)  +\left(  d_{it}-d_{jt}+d_{kt}-\delta_{jt}\right)  ,
\]
$1\leq n\leq m,1\leq t\leq3,$where $i=i_{n+1},j=j_{n+1},$ $k=k_{n+1}$ and
$\delta_{rt}$ is the Kronecker delta.
\end{proposition}

\begin{proof}
Let $D_{n}=\left[  d_{rt}\right]  \ $and $D^{\prime}=D_{n+1}$ $=\left[
d_{rt}^{\prime}\right]  $ the data matrices of steps $n$ and $n+1$ in the
transformation process; besides let $i=i_{n+1},j=j_{n+1}\ $be the guides,
$k=k_{n+1}$ be the idle {\normalsize string}$\ $in step $n+1$ and $a=a_{n+1}$.
By (\ref{reldprima}) we have,%
\begin{align}
s_{n+1,k} &  =d_{ik}^{\prime}+d_{jk}^{\prime}+d_{kk}^{\prime}=\left(
1+2\left\lfloor \frac{a}{2}\right\rfloor +2\left\lceil \frac{a}{2}\right\rceil
\right)  d_{ik}+\\
&  \left(  1+2\left\lceil \frac{a-1}{2}\right\rceil +2\left\lfloor \frac
{a-1}{2}\right\rfloor \right)  d_{jk}+d_{kk},\nonumber\\
s_{n+1,i} &  =d_{ii}^{\prime}+d_{ji}^{\prime}+d_{ki}^{\prime}=\left(
1+2\left\lfloor \frac{a}{2}\right\rfloor +2\left\lceil \frac{a}{2}\right\rceil
\right)  d_{ii}+\nonumber\\
&  \left(  1+2\left\lceil \frac{a-1}{2}\right\rceil +2\left\lfloor \frac
{a-1}{2}\right\rfloor \right)  d_{ji}+d_{ki}+\left\lfloor \frac{a}%
{2}\right\rfloor +\left\lceil \frac{a}{2}\right\rceil ,\nonumber\\
s_{n+1,j} &  =d_{ij}^{\prime}+d_{jj}^{\prime}+d_{kj}^{\prime}=\left(
1+2\left\lfloor \frac{a}{2}\right\rfloor +2\left\lceil \frac{a}{2}\right\rceil
\right)  d_{ij}+\nonumber\\
&  \left(  1+2\left\lceil \frac{a-1}{2}\right\rceil +2\left\lfloor \frac
{a-1}{2}\right\rfloor \right)  d_{jj}+d_{kj}+\left\lceil \frac{a-1}%
{2}\right\rceil +\left\lfloor \frac{a-1}{2}\right\rfloor ,~\nonumber
\end{align}
but as $1+2\left\lfloor \dfrac{a}{2}\right\rfloor +2\left\lceil \dfrac{a}%
{2}\right\rceil =1+2a$ and $\left\lfloor \dfrac{a}{2}\right\rfloor
+\left\lceil \dfrac{a}{2}\right\rceil =a$ we have
\begin{equation}
s_{n+1,t}=\left\{
\begin{array}
[c]{ll}%
\left(  1+2a\right)  d_{ik}+\left(  2a-1\right)  d_{jk}+d_{kk}, & \text{if
}t=k,\\
\left(  1+2a\right)  d_{ii}+\left(  2a-1\right)  d_{ji}+d_{ki}+a, & \text{if
}t=i,\\
\left(  1+2a\right)  d_{ij}+\left(  2a-1\right)  d_{jj}+d_{kj}+a-1, & \text{if
}t=j,
\end{array}
\right.
\end{equation}
therefore%
\begin{align*}
s_{n+1,t} &  =\left(  1+2a\right)  d_{it}+\left(  2a-1\right)  d_{jt}%
+d_{kt}+a\delta_{it}+\left(  a-1\right)  \delta_{jt}\\
&  =a\left(  d_{it}+d_{jt}+d_{kt}\right)  +a\left(  d_{it}+d_{jt}%
-d_{kt}+\delta_{it}+\delta_{jt}\right)  +\left(  d_{it}-d_{jt}+d_{kt}%
-\delta_{jt}\right)  \\
&  =as_{n,t}+a\left(  d_{it}+d_{jt}-d_{kt}+\delta_{it}+\delta_{jt}\right)
+\left(  d_{it}-d_{jt}+d_{kt}-\delta_{jt}\right)  .
\end{align*}

\end{proof}

Now we need to find expressions for $\left(  d_{it}+d_{jt}-d_{kt}+\delta
_{it}+\delta_{jt}\right)  $ and $\left(  d_{it}-d_{jt}+d_{kt}-\delta
_{jt}\right)  $. The following lemmas will take care of finding these
expressions that will be used in proving the main result.

\begin{lemma}
\label{lemdifij}If $i=i_{n+1},j=j_{n+1},k=k_{n+1}$ and $a=a_{n+1},$ then
\begin{equation}
d_{it}^{\prime}-d_{jt}^{\prime}=\left(  -1\right)  ^{\mu_{a}}\left(
d_{it}+d_{jt}+\mu_{a}\delta_{it}+\left(  1-\mu_{a}\right)  \delta_{jt}\right)
,\ \ \text{for }1\leq t\leq3\text{,} \label{ecdifij}%
\end{equation}
where $\mu_{a}$ is given by (\ref{defmu}).
\end{lemma}

\begin{proof}
Suppose $i=i_{n+1},j=j_{n+1},$ $k=k_{n+1}$ and $a=a_{n+1}$. Using the
relations given in (\ref{reldprima}) we get%
\begin{equation*}
d_{ik}^{\prime}-d_{jk}^{\prime}=\left(  1+2\left\lfloor \frac{a}%
{2}\right\rfloor -2\left\lceil \frac{a}{2}\right\rceil \right)  d_{ik}+\left(
-1+2\left\lceil \frac{a-1}{2}\right\rceil -2\left\lfloor \frac{a-1}%
{2}\right\rfloor \right)  d_{jk},%
\end{equation*}%
\begin{align*}
d_{ii}^{\prime}-d_{ji}^{\prime}  & =\left(  1+2\left\lfloor \frac{a}%
{2}\right\rfloor -2\left\lceil \frac{a}{2}\right\rceil \right)  d_{ii}+\left(
-1+2\left\lceil \frac{a-1}{2}\right\rceil -2\left\lfloor \frac{a-1}%
{2}\right\rfloor \right)  d_{ji}+\\
& \left(\left\lfloor \frac{a}{2}\right\rfloor -\left\lceil \frac{a}{2}\right\rceil
\right),\nonumber
\end{align*}
\begin{align*}
d_{ij}^{\prime}-d_{jj}^{\prime} &  =\left(  1+2\left\lfloor \frac{a}%
{2}\right\rfloor -2\left\lceil \frac{a}{2}\right\rceil \right)  d_{ij}+\left(
-1+2\left\lceil \frac{a-1}{2}\right\rceil -2\left\lfloor \frac{a-1}%
{2}\right\rfloor \right)  d_{jj}+\\
&  \left(\left\lceil \frac{a-1}{2}\right\rceil -\left\lfloor \frac{a-1}%
{2}\right\rfloor \right),\nonumber
\end{align*}
but$\ \left\lfloor \frac{a}{2}\right\rfloor -\left\lceil \frac{a}%
{2}\right\rceil =-\mu_{a}=\left(  -1\right)  ^{\mu_{a}}\mu_{a}%
,\ \ \ \left\lceil \frac{a-1}{2}\right\rceil -\left\lfloor \frac{a-1}%
{2}\right\rfloor =1-\mu_{a}=\left(  -1\right)  ^{\mu_{a}}\left(  1-\mu
_{a}\right)  $ and $1+2\left\lfloor a/2\right\rfloor -2\left\lceil
a/2\right\rceil =-1+2\left\lceil \frac{a-1}{2}\right\rceil -2\left\lfloor
\frac{a-1}{2}\right\rfloor =\left(  -1\right)  ^{\mu_{a}},$ therefore
\begin{align}
&  d_{ik}^{\prime}-d_{jk}^{\prime}=\left(  -1\right)  ^{\mu_{a}}\left(
d_{ik}+d_{jk}\right)  \\
&  d_{ii}^{\prime}-d_{ji}^{\prime}=\left(  -1\right)  ^{\mu_{a}}\left(
d_{ii}+d_{ji}\right)  -\mu_{a}\nonumber\\
&  d_{ij}^{\prime}-d_{jj}^{\prime}=\left(  -1\right)  ^{\mu_{a}}\left(
d_{ij}+d_{jj}\right)  +1-\mu_{a}.\nonumber
\end{align}
Also, as%
\[
\left(  -1\right)  ^{\mu_{a}}\left(  \mu_{a}\delta_{it}+\left(  1-\mu
_{a}\right)  \delta_{jt}\right)  =\left\{
\begin{array}
[c]{cc}%
0, & \text{if }t=k\\
-\mu_{a}, & \text{if }t=i\\
1-\mu_{a}, & \text{if }t=j
\end{array}
\right.
\]
we get (\ref{ecdifij}).
\end{proof}

The following corollary will be used in the proof of Theorem \ref{theomain}.
Later, we will consider only the values of the sum of only two of the columns
of $D_{n}$, because the other one will give redundant information.

\begin{corollary}
\label{corsn}If $i=i_{n+1},j=j_{n+1},$ $k=k_{n+1}$and $a=a_{n+1}\ $then%
\begin{equation}
s_{n,t}=\left(  -1\right)  ^{\mu_{a}}(d_{it}^{\prime}-d_{jt}^{\prime}%
)+d_{kt}^{\prime}-\mu_{a}\delta_{it}+\left(  \mu_{a}-1\right)  \delta
_{jt},\ \ \ \text{for\ \ }1\leq t\leq3, \label{eccorsn}%
\end{equation}
where $s_{n,t}=d_{1t}+d_{2t}+d_{3t}.$
\end{corollary}

\begin{proof}
It is enough to recall that $k=k_{n+1}$ is the idle {\normalsize string}, so
$d_{kt}^{\prime}=d_{kt}$, therefore from Lemma \ref{lemdifij} we get,%
\[
\left(  -1\right)  ^{\mu_{a}}\left(  d_{it}^{\prime}-d_{jt}^{\prime}\right)
+d_{kt}^{\prime}=d_{it}+d_{jt}+d_{kt}+\mu_{a}\delta_{it}+\left(  1-\mu
_{a}\right)  \delta_{jt},\ \ \text{for\ }1\leq t\leq3.
\]

\end{proof}

\begin{lemma}
\label{lemsnmenos1}If $i=i_{n+1},j=j_{n+1}$ and $k=k_{n+1}$ then
\begin{equation}
s_{n-1,t}=d_{it}-d_{jt}+d_{kt}-\delta_{jt} \label{relsmen}%
\end{equation}

\end{lemma}

\begin{proof}
We apply Corollary \ref{corsn} to the step $n$.

For $a_{n}$ even, $\mu_{a_{n}}=0$ and by Corollary \ref{reglaguias},
$i_{n}=k_{n+1}=k,$ $j_{n}=j_{n+1}=j$ and $k_{n}=i_{n+1}=i,\ $so we get by
(\ref{eccorsn})
\[
s_{n-1,t}=\left(  -1\right)  ^{\mu_{a_{n}}}(d_{kt}-d_{jt})+d_{it}-\mu_{a_{n}%
}\delta_{kt}+\left(  \mu_{a_{n}}-1\right)  \delta_{jt}=d_{kt}-d_{jt}%
+d_{it}-\delta_{jt}.
\]
For $a_{n}$ odd, $\mu_{a_{n}}=1$ and by Corollary \ref{reglaguias}
$i_{n}=j_{n+1}=j,$ $j_{n}=k_{n+1}=k$ and $k_{n}=i_{n+1}=i$,\ so by
(\ref{eccorsn})%
\[
s_{n-1,t}=\left(  -1\right)  ^{\mu_{a_{n}}}(d_{jt}-d_{kt})+d_{it}-\mu_{a_{n}%
}\delta_{jt}+\left(  \mu_{a_{n}}-1\right)  \delta_{kt}=-d_{jt}+d_{kt}%
+d_{it}-\delta_{jt},
\]
therefore, in both cases we get (\ref{relsmen}).
\end{proof}

This lemma tell us that the values of $s_{n-1,t}$ can be computed using the
matrix $D_{n}$.

\begin{lemma}
\label{lemdk}If $i=i_{n+1},j=j_{n+1},k=k_{n+1}$ then
\begin{equation}
d_{kt}=d_{it}+d_{jt}-\delta_{kt}+\delta_{3t}+\left(  -1\right)  ^{\delta_{3t}%
}\mu_{n},\ \ \ \ 1\leq t\leq3, \label{conlemdk}%
\end{equation}
where$\ \mu_{n}=0$ if $n$ is even and$\ \mu_{n}=1$ if $n$ is odd.
\end{lemma}

\begin{proof}
The proof is by induction on $n$. Notice that condition (\ref{conlemdk})
corresponds to the following three conditions:%
\begin{align*}
d_{k1}  &  =d_{i1}+d_{j1}-\delta_{k1}+\mu_{n},\\
d_{k2}  &  =d_{i2}+d_{j2}-\delta_{k2}+\mu_{n}\\
d_{k3}  &  =d_{i3}+d_{j3}-\delta_{k3}+1-\mu_{n}.
\end{align*}

For $n=0,$ $D_{0}=\left[  0\right]  ,$ and the guides for step one are
$i=i_{1}=1,j=j_{1}=2$ and the idle string is $k=k_{1}=3$. Then%
\begin{align*}
d_{31}  &  =0=d_{11}+d_{21}-\delta_{31}+\mu_{0},\\
d_{32}  &  =0=d_{12}+d_{22}-\delta_{32}+\mu_{0},\\
d_{33}  &  =0=d_{13}+d_{23}-\delta_{33}+1-\mu_{0}.
\end{align*}

For $n=1,$
\begin{equation}
D_{1}=%
\begin{bmatrix}
\left\lfloor \dfrac{a_{1}}{2}\right\rfloor _{{}} & \left\lceil \dfrac{a_{1}%
-1}{2}\right\rceil  & 0\\
\left\lceil \dfrac{a_{1}}{2}\right\rceil ^{{}} & \left\lfloor \dfrac{a_{1}%
-1}{2}\right\rfloor  & 0\\
0 & 0 & 0
\end{bmatrix}
. \label{matd1}%
\end{equation}
If $a_{1}$ is even, by Corollary \ref{reglaguias}, $i=i_{2}=3,j=j_{2}=2$ and
$k=k_{2}=1,$ then
\begin{align*}
d_{11}  &  =\left\lfloor \frac{a_{1}}{2}\right\rfloor =\left\lceil \frac
{a_{1}}{2}\right\rceil =d_{31}+d_{21}-\delta_{11}+\mu_{1},\\
d_{12}  &  =\left\lceil \frac{a_{1}-1}{2}\right\rceil =\left\lfloor
\frac{a_{1}-1}{2}\right\rfloor +1=d_{32}+d_{22}-\delta_{12}+\mu_{1},\\
d_{13}  &  =0=d_{33}+d_{23}-\delta_{13}+1-\mu_{1}.
\end{align*}
If $a_{1}$ is odd, by Corollary \ref{reglaguias}, $i=i_{2}=3,j=j_{2}=1$ and
$k=k_{2}=2$ therefore%
\begin{align*}
d_{21}  &  =\left\lceil \frac{a_{1}}{2}\right\rceil =\left\lfloor \frac{a_{1}%
}{2}\right\rfloor +1=d_{31}+d_{11}+1=d_{31}+d_{11}-\delta_{21}+\mu_{1},\\
d_{22}  &  =\left\lfloor \frac{a_{1}-1}{2}\right\rfloor =\left\lceil
\frac{a_{1}-1}{2}\right\rceil =d_{32}+d_{12}-\delta_{22}+\mu_{1},\\
d_{13}  &  =0=d_{33}+d_{23}-\delta_{23}+1-\mu_{1}.
\end{align*}
Suppose that $i=i_{n+1},\ j=j_{n+1},k=k_{n+1}$ and that relations
(\ref{conlemdk}) hold for step $n$. To prove that the results are valid for
step $n+1$ we need to find the $n+2$ step guides, that depend on wether
$a_{n+1}$ is even or odd.

If $a_{n+1}$ is even by Corollary \ref{reglaguias}, $i_{n+2}=k_{n+1}%
=k,j_{n+2}=j_{n+1}=j$ and $k_{n+2}=i_{n+1}=i$ and we need to prove that
\[
d_{it}^{\prime}=d_{jt}^{\prime}+d_{kt}^{\prime}-\delta_{i1}+\delta
_{3t}+\left(  -1\right)  ^{\delta_{3t}}\mu_{n+1},\text{ for}~1\leq t\leq3.
\]
By Lemma \ref{lemdifij} we have
\[
d_{it}^{\prime}-d_{jt}^{\prime}=d_{it}+d_{jt}+\delta_{jt}%
\]
As $k$ is the idle {\normalsize string} in the $n+1$ step, we have that
$d_{kt}^{\prime}=d_{kt}$ and for (\ref{conlemdk}) we get%
\begin{equation}
d_{it}+d_{jt}=d_{kt}^{\prime}+\delta_{kt}-\delta_{3t}-\left(  -1\right)
^{\delta_{3t}}\mu_{n}, \label{ecdidj}%
\end{equation}
so%
\[
d_{it}^{\prime}=d_{jt}^{\prime}+\left(  d_{it}+d_{jt}\right)  +\delta
_{jt}=d_{jt}^{\prime}+\left(  d_{kt}^{\prime}+\delta_{kt}-\delta_{3t}-\left(
-1\right)  ^{\delta_{3t}}\mu_{n}\right)  +\delta_{jt},
\]
but $\delta_{kt}+\delta_{jt}=1-\delta_{it}$ and
\begin{equation}
-\delta_{3t}-\left(  -1\right)  ^{\delta_{3t}}\mu_{n}=-\delta_{3t}-\left(
-1\right)  ^{\delta_{3t}}\left(  1-\mu_{n+1}\right)  =\delta_{3t}-1+\left(
-1\right)  ^{\delta_{3t}}\mu_{n+1}, \label{mnmas1}%
\end{equation}
so%
\[
d_{it}^{\prime}=d_{jt}^{\prime}+d_{kt}^{\prime}+1-\delta_{it}-\delta
_{3t}-\left(  -1\right)  ^{\delta_{3t}}\mu_{n}=d_{jt}^{\prime}+d_{kt}^{\prime
}-\delta_{i1}+\delta_{3t}+\left(  -1\right)  ^{\delta_{3t}}\mu_{n+1}.
\]
If $a=a_{n+1}$ is odd, by Corollary \ref{reglaguias} $i_{n+2}=k_{n+1}%
=k,j_{n+2}=i_{n+1}=i$ and $k_{n+2}=j_{n+1}=j.$

By Lemma \ref{lemdifij} we have
\[
d_{it}^{\prime}-d_{jt}^{\prime}=-d_{it}-d_{jt}-\delta_{it}%
\]
thus, by (\ref{ecdidj}) and (\ref{mnmas1}) we have
\begin{align*}
d_{jt}^{\prime}  &  =d_{it}^{\prime}+(d_{it}+d_{jt})+\delta_{it}%
=d_{it}^{\prime}+(d_{kt}^{\prime}+\delta_{kt}-\delta_{3t}-\left(  -1\right)
^{\delta_{3t}}\mu_{n})+\delta_{it}\\
&  =d_{it}^{\prime}+d_{kt}^{\prime}+\delta_{kt}+\delta_{it}-1+\delta
_{3t}+\left(  -1\right)  ^{\delta_{3t}}\mu_{n+1}\\
&  =d_{it}^{\prime}+d_{kt}^{\prime}+\delta_{jt}+\delta_{3t}+\left(  -1\right)
^{\delta_{3t}}\mu_{n+1}.
\end{align*}

\end{proof}

\begin{corollary}
\label{cor2dk}If $k=k_{n+1}$ then
\[
s_{n,t}=2d_{kt}+\delta_{kt}-\delta_{3t}-\left(  -1\right)  ^{\delta_{3t}}%
\mu_{n},\ \ \ \ 1\leq t\leq3.
\]

\end{corollary}

\section{Proof of Main Result}

Now we use the results of the previous section to prove the main result:

\begin{theorem}
[Main Result]\label{theomain}Given the continued fraction $\left[  a_{1}%
,a_{2},\cdots a_{m}\right]  $ that represents the rational $p/q,$ the integers
$p_{n}$ and $q_{n}$ defined in Definition \ref{defpyq} satisfy the recurrence
relations
\begin{align*}
p_{n+1}  &  =a_{n+1}p_{n}+p_{n-1},\ \text{for\ }1\leq n,\ \ p_{0}%
=1,\ p_{1}=a_{1},\\
q_{n+1}  &  =a_{n+1}q_{n}+q_{n-1},\ \text{for\ }1\leq n,\ \ q_{0}=0,q_{1}=1.
\end{align*}

Therefore $p_{n}/q_{n}$ is the $n$-convergent of the continued fraction
$\left[  a_{1},\cdots,a_{m}\right]  $ and so
\[
p_{m}/q_{m}=p/q\text{,}%
\]
therefore the rational link $C\left[  a_{1},a_{2},\cdots a_{m}\right]  $ is
equivalent to the 2-brigde link $p/q$.
\end{theorem}

\begin{proof}
Let $\left[  a_{1},\cdots,a_{m}\right]  $ be a continued fraction. For $n=0,$
$D_{0}=\left[  0\right]  ,$ so $p_{0}=1$ and $q_{0}=0.$ From (\ref{matd1}),
$s_{1,1}=a_{1}$ and $s_{3,1}=0,$ so $p_{1}=a_{1}$ and $q_{1}=1.$

By Proposition \ref{propsum}%
\[
s_{n+1,t}=a_{n+1}s_{n,t}+a_{n+1}\left(  d_{it}+d_{jt}-d_{kt}+\delta
_{it}+\delta_{jt}\right)  +\left(  d_{it}-d_{jt}+d_{kt}-\delta_{jt}\right)  \
\]
If $i=i_{n+1},j=j_{n+1},k=k_{n+1}$ by Lemma \ref{lemdk}
\[
0=d_{it}+d_{jt}-d_{kt}-\delta_{kt}+\delta_{3t}+\left(  -1\right)
^{\delta_{3t}}\mu_{n},
\]
but $\delta_{it}+\delta_{jt}=1-\delta_{kt},$ so
\[
d_{it}+d_{jt}-d_{kt}-\delta_{kt}=d_{it}+d_{jt}-d_{kt}+\delta_{it}+\delta
_{jt}-1=-\delta_{3t}-\left(  -1\right)  ^{\delta_{3t}}\mu_{n}%
\]
and by Lemma \ref{lemsnmenos1}
\[
s_{n-1,t}=d_{it}-d_{jt}+d_{kt}-\delta_{jt}%
\]
so
\begin{equation}
s_{n+1,t}=a_{n+1}s_{n,t}+a_{n+1}\left(  1-\delta_{3t}-\left(  -1\right)
^{\delta_{3t}}\mu_{n}\right)  +s_{n-1,t} \label{ecrec1}%
\end{equation}
then, for $t=1,$ $s_{n,1}=p_{n}-1+\mu_{n}$%
\[
p_{n+1}-1+\mu_{n+1}=a_{n+1}\left(  p_{n}-1+\mu_{n}\right)  +a_{n+1}(1-\mu
_{n})+p_{n-1}-1+\mu_{n-1}%
\]
and $\mu_{n+1}=\mu_{n-1}$ so
\[
p_{n+1}=a_{n+1}p_{n}+p_{n-1}%
\]
For $t=3,$ $s_{n,3}=q_{n}-\mu_{n},$ so (\ref{ecrec1}) becomes%
\[
q_{n+1}-\mu_{n+1}=a_{n+1}\left(  q_{n}-\mu_{n}\right)  +a_{n+1}(1-1+\mu
_{n})+q_{n-1}-\mu_{n-1}%
\]
therefore%
\[
q_{n+1}=a_{n+1}q_{n}+q_{n-1}.
\]

\end{proof}

\section*{Acknowledgments}
Thanks to Universidad Nacional de Colombia, Sedes Medell\'{\i}n y Manizales.


\begin{thebibliography}{0}


\bibitem {BuZi}G. Burde and H. Zieschang, \textit{Knots}, 2$^{{}}$ ed, (Walter
de Gruyter, New York, NY, 2003).

\bibitem {Bu}D. Burton. \textit{Elementary Number Theory}, 7$^{{}}$ ed.,(
Mcgraw-Hill, 2011).

\bibitem {Co}J. Conway, An enumeration of knots and links and some of their
related properties, In \textit{Computational Problems in Abstract Algebra,
Proc. Conf. Oxford} (1967), 329-358.

\bibitem {Gr}H. Gruber, Rational Knots,
\textit{www2.tcs.ifi.lmu.de/\symbol{126}gruberh/}

\bibitem {HMTT}H. Hilden, J. Montesinos, D. Tejada and M. Toro, On the
classification of 3-bridge links, \textit{Revista Colombiana de
Matem\'{a}ticas}, \textbf{46 }(2)\textbf{ (}2012\textbf{) }113-144\textbf{.}

\bibitem {klt}L. Kauffman and S. Lambropoulou, On the Classification of
Rational Tangles,\textit{ Adv. Appl. Math}. \textbf{33}(2)\textbf{ }(2004), 199-237.

\bibitem {kl}L. Kauffman and S. Lambropoulou, On the Classification of
Rational Knots, \textit{L' Enseign. Math.}, \textbf{49}(3-4) (2003), 357-410.

\bibitem {Ka}Kawauchi, \textit{A survey of Knot Theory, (}Birkh\"{a}user,
Basel, 1996).

\bibitem {Mu}K. Murasugi, \textit{Knot theory and its applications,
(}Birkh\"{a}user, Boston, 1996).

\bibitem {Sc}H. Schubert, Knoten mit zwei Br\"{u}cken\textit{, (}Math. Z.)
\textbf{65 }(1956), 133-170

\bibitem {Wi}D. De Wit, The 2-bridge knots of up to 16 crossings,
\textit{Journal of Knot Theory and Its Ramifications}, Vol. 16, No. 8 (2007) 997--1019

\end{thebibliography}
\end{document}